\documentclass[letterpaper,11pt,reqno]{amsart}

\usepackage{amsmath, amssymb, amscd, MnSymbol,mathrsfs}

\newtheorem{theorem}{Theorem}[section] 

\newtheorem{prop}[theorem]{Proposition}

\theoremstyle{plain}
\newtheorem{example}{Example}

\theoremstyle{definition}
\newtheorem{definition}{Definition}[section] 

\theoremstyle{remark}
\newtheorem{remark}{Remark}[section] 

\DeclareMathAlphabet{\mathpzc}{OT1}{pzc}{m}{it}

\newcommand{\LogL}{\mathsf{L}}
\newcommand{\ruleR}{\mathsf{r}}
\newcommand{\bydef} {:=} 
\newcommand{\lbr}{\langle}
\newcommand{\rbr}{\rangle}
\newcommand{\mpos}{\oplus}
\newcommand{\mrej}{\ominus}
\newcommand{\mnu}{\odot}
\newcommand{\notmnu}{\overline{\odot}}

\newcommand{\Stms}{\mathcal{S}}

\newcommand{\mAStms}{\mathcal{S}_a}
\newcommand{\mand}{\wedgedot}
\newcommand{\mor}{\veedot}
\newcommand{\mimpl}{\dot{\to}}
\newcommand{\mneg}{\dotminus}

\newcommand{\mtop}{\downvdash}
\newcommand{\mbot}{\perp}
\newcommand{\impl}{\rightarrow}
\newcommand{\emptyd}{\mathrm{V}}

\newcommand{\SF}{\mathbf{S4}}
\newcommand{\Vars}{\mathcal{P}}
\newcommand{\Frm}{\mathsf{Fm}}
\newcommand{\CFin}{\boldsymbol{\mathsf{F}}}
\newcommand{\Ax}{\mathpzc{Ax}}
\newcommand{\Rules}{\mathsf{R}}
\newcommand{\DS}{\mathsf{S}}

\newcommand{\class}[1]{\mathcal{#1}}
\newcommand{\Logl}{\mathcal{L}} 
\newcommand{\means}{\leftrightharpoons}
\newcommand{\eqv}{\leftrightarrow}

\def\alg#1{\mathsf{#1}}
\def\Log#1{\mathscr{#1}}
\newcommand{\set}[2]{\{#1 \mid #2\}}

\begin{document}

\title{A Meta-Logic of Inference Rules: Syntax}

\author{Alex Citkin}

\address{Alex Citkin \\(Metropolitan Telecommunications, New York) \\30 Upper Warren Way Warren, NJ 07059}
\email{acitkin@gmail.com}

\begin{abstract}
 This work was intended to be an attempt to introduce the meta-language for working with multiple-conclusion inference rules that admit asserted propositions along with the rejected propositions. The presence of rejected propositions, and especially the presence of the rule of reverse substitution, requires certain change the definition of structurality.  
\end{abstract}


\maketitle

\section{Introduction}
\subsection{Motivation}

The idea of a meta-logic as introduced in the present paper was triggered by the following observations.

\textbf{Observation (A).} If $\LogL$ is a (propositional) logic and $\ruleR \bydef \Gamma/A$ is a structural (inference) rule, we say that $\ruleR$ is admissible in the logic $\LogL$ if any substitution that makes all premises from $\Gamma$ valid in $\LogL$, makes $A$ valid too. A logic $\LogL$ is said to be closed under (applications of) rule $\ruleR$ if the rule $\ruleR$ allows to derive from $\LogL$ only the formulas from $\LogL$. It is not hard to see that a rule $\ruleR$ is admissible in $\LogL$ if and only if $\LogL$ is closed under $\ruleR$. If we try to apply the definition of admissibility to multiple conclusion rules, the equivalence is not necessarily true. More precisely, a multiple-conclusion rule $\ruleR \bydef \Gamma/\Delta$ may be not admissible in $\LogL$ (there is a substitution that makes all premises valid and all conclusions not valid in $\LogL$), and yet $\LogL$ is closed under $\ruleR$: there are the deductive systems having CPL as a set of theorems and containing $\ruleR$ is an inference rule. An example is very simple and somewhat unexpected: a rule $A \lor B/A,B$ (the disjunction property) is, certainly, not admissible in the classical propositional logic (CPL thereafter), and nevertheless, if we add this rule to the rules of substitution and Modus Ponens, the logic will remain consistent. Hence, due to CPL being Post complete, the added rule does not change the logic and, therefore, does not allow us to derive from CPL any formulas that are not in CPL.       

\textbf{Observation (B).} The admissible rules are closely related to the refutation rules (in the sense of {\L}ukasiewicz, see e.g. \cite{Slupecki_Bryll_Wybraniec_1971,Slupecki_Bryll_Wybraniec_1972}): if a rule $A_1,\dots,A_n/B$ is admissible  in a given logic $\LogL$, the rule $\dashv B/\dashv A_1,\dots,\dashv A_n$ is a refutation rule for $\LogL$. So, in order to be able to study both types of rules at the same time, we need a proper framework, namely, we need to use the rules containing asserted formulas and rejected formulas. In turn, this will allow us to extend the notion of admissible rule to this kind of generalized rules. 

\textbf{Observation (C).} The use of multiple-conclusion rules in an inference, which essentially is a proof by cases, requires a better understanding under which condition we can eliminate a case/alternative. Indeed, in a proof by cases - if we want to prove a formula $A$ - we first arrive at a complete set of possible cases/alternatives $A_1,\dots,A_n$ and then we consider each case $A_i$ separately trying either to derive $A$ from $A_i$ or to show that the case $A_i$ is impossible by deriving from $A_i$ some kind of contradiction (semantically or syntactically). So, when we consider an alternative, we are either trying to derive a target formula, or to arrive at some kind of contradiction, and the latter lets us eliminate this alternative. If we deal with a formal proof, we are not able to employ the semantical means. Thus, in order to demonstrate a contradiction we are often trying to derive a negation of a formula that was proven earlier. But in the case when the language does not have a negation (or a constant for false) it is impossible. The alternative approach is to try and arrive at a contradiction by syntactically deriving the refutability of a formula that we already proved. This is another reason why we need to have the means for proving refutability as a part of our proof system.   

If we want to define a logic syntactically we employ a notion of a deductive system, and we understand this logic as a set of formulas derivable in a given deductive system. Often, we define inconsistency of a logic, or, more precisely, inconsistency of a deductive system, as an ability to derive any formula. Roughly speaking, abolishing the rule that allows to derive every formula from some form of contradiction allows us to deal with paraconsistent logics.

\textbf{The suggested approach.} Let us assign equal rights to asserted and rejected propositions\footnote{The idea to consider rejected propositions can be traced back as early as 1940-th to books by R.~Carnap (for more historical details and references see Section \ref{hist}).}. Let us consider a logic as a pair $\lbr \LogL^+, \LogL^-\rbr$ of sets of formulas, where $\LogL^+$ is a set of asserted propositions (theorems) and $\LogL^-$ is a set of rejected propositions (anti-theorems). Thus, at least in theory, we may have the case when a formula $A$ is asserted and rejected as well as the case when a formula $B$ is neither asserted, nor rejected. The first case is related to inconsistency/paraconsistency, while the second case is related to insufficiency of information. And a logic understood as such a pair represents, perhaps, better the situation with requests to a database when the responses to a request can be ambivalent.      

If we assign equal rights to the asserted and rejected propositions, it is natural to include both types of proposition into inference and consider the inference rules containing asserted propositions as well as rejected propositions. It means that we define the inference rules not on the sets of formulas, but rather on the sets of statements like $\mpos A$ - formula $A$ is asserted and $\mrej A$ - formula $A$ is rejected. Thus, we derive not a formula, but a statement about the formula. For instance, Modus Tollens can be represented in the following way $\mpos (A \impl B), \mrej B/ \mrej A$ (in the {\L}ukasiewicz's notation it would be $\vdash(A \to B),\dashv B/\dashv A$ , but we reserve $\vdash$ for a different use). Let us note that an ability to derive $\mpos A$ and $\mrej A$ for some formula $A$, represents some kind of inconsistency, while an ability to derive $\mpos A$ and $\mpos \neg A$ may not lead to inconsistency, especially if we do not accept a rule $\mpos \neg A/\mrej A$.

In order to make our work with statements easier, we construct a meta-logic. One of the challenges of using multiple-conclusion rules is definition a notion of inference. In \cite{{Kneale_Province_1956},Shoesmith_Smiley_Book} the reader can see how challenging it gets. In the proposed meta-logic, the inference can be defined as linear: with a rule $\Gamma/\Delta$ in the meta-logic we associate a statement $\mand \Gamma \mimpl \mor \Delta$, where $\mand,\mimpl,\mor$ are meta-connectives, and we use this statement in an inference as we normally use a formula (from more details see Section \ref{inference}).    

The differences between single- and multiple-conclusion rules lead to necessity to clarify such important notions as structurality, inference, closure, etc. The meta-logic that we are introducing is the first step in this direction.

\section{History of the Subject} \label{hist}

The idea of including refutation into inference process, as well as the idea that led to introducing the multiple conclusion rules and logics, came from the books \cite{Carnap_Introduction_1942,Carnap_Formalization_1943} by Carnap. Then, for quite some time, these two lines of research were conducted  independently by different researchers. Let us briefly look at the history of the subject and at the motivations that instigated different researchers. 

\subsection{R.~Carnap} \label{Carnap}

\subsubsection{Refutation} \label{refutation}

The idea to include refutation into inference process can be traced back at least to R.~Carnap\footnote{Some ideas of including refutation into syllogistic were introduced by Aristotle.}. Carnap was considering the calculi together with their interpretations: "Although the rules of a calculus do not speak about interpretations, they are nevertheless practically meant in such a way as to restrict possible interpretations" \cite[Paragraph 28]{Carnap_Introduction_1942}. He also observed that there are non-normal interpretations. And, according to Carnap, there are two kinds of non-normal interpretations: the ones that violate the law of contradiction by admitting a proposition and its negation, and the ones that violate the law of excluded middle by admitting neither a proposition, nor its negation is true. In order to exclude the first kind of non-normal interpretations Carnap suggested to add refutations to the calculus: "One new  syntactical concept which might be added to those used in customary calculi is 'C-false'. It is defined on the basis of 'directly C-false', which is defined by rules of refutation. By adding a rule of this kind to PC, the non-normal interpretations of the first kind can be excluded" (see \cite[Paragraph 20]{Carnap_Formalization_1943}).
And he has extended the notion of a calculus in the following way: "A syntactical system or \textit{calculus} $K$ is a system of formal rules. It consists of a classification of signs, the \textit{rules of formation} (defining 'sentence in $K$'), and the \textit{rules of deduction}. The rules of deduction usually consist of primitive sentences and rules of inference (defining 'directly derivable in $K$'). Sometimes $K$ also contains rules of refutation (defining 'directly refutable in $K$'). If $K$ contains definitions, they may be regarded as additional rules of deduction"( see \cite[Paragraph 24]{Carnap_Introduction_1942}). 

Carnap has also made the following important observation: "And in nearly all or perhaps all of the few calculi where rules of refutation are given, 'directly C-false' ('directly refutable') applies only to sentences or sentential classes from which every sentence is derivable" (see \cite[Paragraph 28]{Carnap_Introduction_1942}). That is, prior to Carnap, a formula would be considered refutable if adding this formula to axioms leads to inconsistency. As we shall see later in the Section , in many cases a formula is considered to be refutable, if its negation is provable. 

Carnap rejected such an approach to refutation (when a proposition is refuted if its negation is provable) on the grounds that this would not exclude the non-normal interpretations: think about single-element matrix, for instance (cf. \cite{Church_Review_1953}). His approach is different: a proposition is C-false (refutable) if an anti-axiom (i.e. a proposition accepted as directly refutable) can be derived from it \cite[D28.3]{Carnap_Introduction_1942}. And he adds a single anti-axiom $\emptyd$ (a constant 'False') to a calculus .

\subsubsection{Junctives} \label{junctives}

In \cite{Carnap_Formalization_1943} Carnap introduced the notions that later became known as multiple-conclusion logics (see, for instance, \cite{Shoesmith_Smiley_Book}). If adding the rules of refutation to a calculi solved a problem with non-normal interpretations of the first kind, adding the junctives solves the problem with non-normal interpretations of the second kind (see \cite[Section D]{Carnap_Formalization_1943}). A \textit{conjunctive} is an ordered pair $\alg{P}^\land$, where $\alg{P}$ is a finite set of propositions and $\alg{P}^\land$ is asserted if every proposition from $\alg{P}$, is asserted. A \textit{disjunctive} is an ordered pair $\alg{P}^\lor$, where $\alg{P}$ is a finite set of propositions and $\alg{P}^\lor$ is asserted if at least one  member of $\alg{P}$ is asserted. Note, that the junctive $\emptyset^\land$ represents constant 'true', while the junctive $\emptyset^\lor$ represents constant 'false'. 

A regular (structural) inference rule $A_1,\dots,A_n/B$ can be regarded as a rule that allows to derive a disjunctive $\{B\}^\lor$ from a conjunctive $\{A_1,\dots,A_n\}^\land$. And Carnap suggests to also consider the rules of form $A/\alg{Q}^\lor$ with disjunctive as a consequence. He used such kind of rule and in \cite[Section E.26]{Carnap_Formalization_1943} where he constructs the calculus $PC^*$ for the classical propositional logic  and which, besides regular axioms and rules, contains the rule 
\begin{equation}
p \lor q/\{p,q\}^\lor. \label{disj}
\end{equation}
He also included a rule for refutation. So, if we consider, for instance, the Boolean algebras as the models for CPL in the Carnap's version, the single-element algebra is not a model due to $\dashv p$ does not hold in it, and all Boolean algebras with more then 2 elements are not the models due to \eqref{disj} does not hold in them.

\begin{remark} The Gentzen's sequents can be viewed as a multipme-conclusion constructions. The following quotation from \cite{Shoesmith_Smiley_Book} explains why Carnap, and not Gentzen, perhaps, should be regarded as the one who introduced multiple-conclusion rules: "Its germ can be found in Gerhard Gentzen's celebrated \textit{Untersuchungen {\"u}ber das logische Schliessen} (1934) if one is prepared to interpret his calculus of 'sequents' as a metatheory for a multiple-conclusion logic, but this is contrary to Gentzen's own interpretation, and it was Rudolf Carnap who first consciously broached the subject in his book \textit{Formalization of logic} (1943)" (cf. also the historical note in \cite[Section 2.1]{Shoesmith_Smiley_Book}).
\end{remark}

\subsection{J.~{\L}ukasiewicz} \label{Lukasiweicz}

In the first edition of his book \cite{Lukasiewicz_Book} in 1951 J.~{\L}ukasiewicz (who, as it appears, was not familiar with Carnap's research)  also included refutation into calculus for CPL. More precisely, he added to a regular classical propositional calculus with rules Modus Ponens (MP) and Substitution (Sb) an anti-axiom $\dashv p$, where $p$ is a (propositional) variable, and two rules: Modus Ponens (MT) and Reverse Substitution (RS)
\[
\dashv \sigma(A)/\dashv A, \text{ for every substitution } \sigma \text{ and formula } A. \tag{RS}
\] 

The {\L}uksiewicz's motivation was totally different from the Carnap's. In his book \cite[Preface]{Lukasiewicz_Book} J.~{\L}ukasiewicz writes: "The most important new results in this part I consider to be the proof of decision, given by my pupil J. S{\l}upecki, and the idea of rejection introduced by Aristotle and applied by myself to theory of deduction." And he added "Modern formal logic, as far as I know, does not use 'rejection' as an operator opposed to Frege's 'assertion'. The rules of rejection are not yet known"\cite[Paragraph 20]{Lukasiewicz_Book}. 

{\L}ukasiewicz had also observed that for the formal system representing Aristotle syllogistic it is not enough to use only the rules of reverse substitution and Modus Tollens for refutation. He wrote \cite[p. 75]{Lukasiewicz_Book} "A new rule of rejection must be added to the system to complete the insufficient characterization of the Aristotelian logic given by the four axioms. This rule was found by J.~S{\l}upecki."(for more about S{\l}upecki rules see \cite{Slupecki_Book_1948,Kulicki_Remarks_2002}) 

In \cite{Lukasiewicz_Intuitionistic_1952} {\L}ukasiewicz suggested that adding $\dashv p$ , RS, MT, $\dashv A, \dashv B / \dashv (A \lor B)$ to the regular axioms and rules of IPC would give the complete refutation system for IPL. He also did not identify rejection with negation. Specifically in \cite{Lukasiewicz_Intuitionistic_1952}, he wrote that "In my recently published work on Aristotle's Syllogistic I gave reasons for introducing "rejection" into classical theory of deduction as a complement of assertion".

Thus, the {\L}ukasiewicz's motivations were (a) to be able to contract a formal system for syllogistic, (b) to have a counterpart for the asserted proposition, (c) to have decidability as a result of axiomatization. 

\subsection{D.~Scott} \label{Scott}

It turned out that the axiomatization suggested by {\L}ukasiewicz for IPL is not complete. This gap was filled by D.~Scott. In \cite{Scott_Completeness_1957} he considered $\to,\top,\bot$-fragment of IPC. A refutation part of his calculus contains an anti-axiom $\dashv \bot$, the rules MT and RS, and some additional (rather complex) rules. The complete axiomatization for IPL that includes refutation was constructed much later by T.~Skura \cite{Skura1989}.

The goal of \cite{Scott_Completeness_1957} (besides fixing the shortcomings of {\L}ukasiewicz's conjecture) was "by use of refutation rules to enumerate the unprovable formulae in much the same way in which we enumerate the theorems, every formula being either provable or refutable. The sets of valid and invalid formulae are closed under the rules. Hence, no formula is both provable and refutable, the calculus is decidable, and a formula is  provable if and only if it is valid." Thus, the motivation here is to construct a calculus that gives decidability and semantical completeness. 

Later, in early 1970-th, D.~Scott introduced the multiple-conclusion consequence relations (see e.g.\cite{Scott_Engendering_1971,Scott_1974}). Nevertheless, he did not include in these relations the rejected propositions.  

\subsection{Refutation: Further development}

The axiomatic systems that include refutation can be roughly divided in three types: dual, complementary and mixed. 

\subsubsection{Dual Systems} A dual system is a system that allows to derive rejected propositions from rejected like we derive asserted propositions from asserted. Similarly to regular logical systems, these systems can be constructed in a form of a closure operator (e.g. \cite{Wojcicki_Dual_1973}). This kind of system have found applications in the computer science and artificial intelligence (e.g. \cite{Tiomkin_Proving_1988}).

\subsubsection{Complementary (Symmetric) Systems} The complementary systems contain, essentially, two subsystems: one for deriving the asserted propositions, and another - for deriving the rejected propositions. For instance, in \cite{Slupecki_Bryll_Wybraniec_1971,Slupecki_Bryll_Wybraniec_1972,Wybraniec_Waldmajer_2011} the authors consider two closure operators: the regular one $Cn$ that gives the theorems, and the complementary one $Cn^-$, that gives anti-theorems (see also \cite[Section 5.2]{Skura_Refutation_2011}). 

\subsubsection{Mixed Systems} In these systems the asserted propositions are derived from the asserted propositions, while the rejected propositions are derived from the rejected and, asserted propositions (think about Modus Tollens, for instance). The system suggested by {\L}ukasiewicz for CPL is, of course, of this type. It is the most common type of the refutation systems. The complete (mixed) systems were constructed for various of logics (see, for instance \cite{Dutkiewicz_Method_1989, Caferra_Zabel_Method_1992,Varzi_Complementary_1992,Goranko_Refutation_1994,Bonatti_Varzi_Meaning_1995,Skura1998,Sochacki_Axiomatic_2007,Caferra_Peltier_Accepting_2008}).

\subsubsection{Direct and Indirect Refutation} R.~Carnap and J.~{\L}ukasiewicz viewed the derivation of a rejection in different ways. 

R.~Carnap suggested that we can use the regular notion of derivation (that allows to derive a formula (proposition) from a set of formulas (propositions), and that we reject a formula $A$ if we can derive from $A$ a formula that we know is rejected, is an anti-axiom. Roughly speaking, if $A \vdash B$ and $B$ is rejected, we reject $A$. We call this approach \textit{indirect} and we call the refutation systems based on this approach \textit{C-system}\footnote{In \cite{Staszek_1971} indirect derivations are called i-derivations.}. For instance, if we add to CPC (with rules Modus Ponens and Substitution) an anti-axiom $\dashv p$ (where $p$ is a propositional variable), then, in this extended calculus, we are able to reject every classically invalid formula, for $p$ is derivable in CPC from every class classically invalid formula.   

J.~{\L}ukasiewicz, on the other hand, suggested that the rejected formulas should be derived from the asserted and rejected formulas by means of regular inference rules (Modus Ponens and Substitution), and by means of additional inference rules that allow to derive the rejected formulas, such rules as Modus Tollens and Reverse Substitution. We call this type of refutation systems \textit{direct} or \textit{{\L}-systems}. For instance, in order to obtain an {\L}-system for the classical logic, {\L}ukasiewicz adds to CPC the same anti-axiom $\dashv p$, but he endows CPC with two new inference rules: Modus Tollens abd Reverse Substitution. 

If we consider only rules Modus Ponens, Substitution, Modus Tollens and Reverse Substitution, not every C-derivation can be converted into {\L}-derivation, and not every C-system can be converted into {\L}-system. In \cite{Staszek_1971} Staszek establishes the conditions under which a C-system can be converted into {\L}-system. Generally speaking, if a propositional language contains implication $\to$ and a formula $p \to (q \to p)$ is derivable, then any C-system can be converted into {\L}-system. For instance, for all intermediate or normal modal logics every C-system can be converted into {\L}-system. If we consider the multiple-conclusion rules with rejected propositions, all C-system become "convertible" into {\L}-systems, and this is an additional reason to use such type rules. 

The following example shows the difference between C- and {\L}-derivations of refutability of $\neg p$ in CPC with rules Modus Ponens (MP) and Substitution (Sb) endowed with an anti-axiom $\dashv p$ and rules Modus Tollens (MT) and Reverse Substitution (RS). 

\begin{table}[h]
\begin{center}
\caption{Example of C- and {\L}-derivation}
\begin{tabular}{l|l|p{2.9cm}||l|p{2.9cm}}
\hline
&\multicolumn{2}{c}{C-derivation}&\multicolumn{2}{c}{{\L}-derivation }\\
\hline
1& $\neg p$ & premiss & $\dashv p$ & anti-axiom\\
2& $\neg (p \to p)$ & from 1 by Sb &$\vdash \neg (p \to p) \to p$&derivable in CPC\\
3& $\neg (p \to p) \to p$& derivable in CPC &$\dashv \neg (p \to p)$& from 1, 2 by MT\\
4& $p$ & by MP from 2, 3 & $\dashv \neg p$ &from 3 by RS\\
5&$\neg p$ is refuted &anti-axiom is derived from $\neg p$ &&\\
\hline
\end{tabular}
\end{center}
\end{table} \label{tbl:example}

\subsection{Multiple-conclusion rules: Further development}

As we mentioned above in the Section \ref{junctives}, the concept of multiple-conclusion rule was introduced by R.~Carnap in \cite{Carnap_Formalization_1943}, where R.~Carnap studied a notion of junctives and rules dealing with junctives. This concept has been developed further by W.~Kneale \cite{Kneale_Province_1956} (some corrections are in \cite{Kneale_1957}) where multiple-conclusion proof was defined. This definition was refined in the 1970-th in the papers by D.~Shoesmith and T.~Smiley and summarized in their book \cite{Shoesmith_Smiley_Book}. At the about the same time multiple-conclusion relations were studied by D.~Scott \cite{Scott_Engendering_1971,Scott_Background_1973, Scott_completeness_1974, Scott_1974}. Let us note that even though D.~Scott considers the relation $\Gamma \vdash \Delta$ where $\Gamma,\Delta$ are finite sets of formulas, he does not regard an expression $\Gamma \vdash \Delta$  for given finite sets of formulas $\Gamma,\Delta$ as an instance of an inference rule: D.~Scott views such an expression as a "conditional statement"  and investigates the means that allow to derive a conditional statement from the set of conditional statements (for more details we refer the reader to \cite{Scott_1974}.) 

In 1999 M.~Kracht in his review \cite{Kracht_Review_1999} suggested to study the admissibility of multiple-conclusion rules, and such a study was carried out, for instance, in \cite{Jerabek_Admissible_2005,Jerabek_Canonical_2009}. Let us also remark that in \cite{Kracht_Judgement_2010} M.~Kracht arrives to a notion of a consequence relation in some respects similar the one introduced in this paper.

\subsection{What This Paper is About} 

In his 1996 paper \cite{Smiley_Rejection_1996} T.~Smiley outlined an approach to combining multiple-conclusion rules with refutability. We are taking a similar approach. We will discuss the similarities and differences later in Section \ref{remarks}. The discussion of philosophical aspects of Smiley's paper can be found in \cite{Johnson_Rejection_1999,Murzi_Hjortland_2009,Incurvati_Smith_2010}. 

In this paper we introduce a metalanguage that gives us the formal syntactic means for working with inference rules (rules thereafter) admitting simultaneously asserted and rejected formulas. In Section \ref{cons} we extend the notion of consequence relation to sets of asserted and rejected propositions (formulas). Then, in the Section \ref{logics}, we define a logics as an ordered pair of a set of asserted and a set of rejected propositions and we introduce the consequence relations and logics that admit rejected propositions independently from asserted propositions, meaning that not asserted proposition does not have to be rejected and vice versa. Then, in the same section, we construct the meta-language and we define inference (derivation). In Section \ref{inference} we will see how a logic understood as a pair can be defined by a deductive system. In the present paper we will not discuss the semantic of such systems.

Let us note that the presence of rejected formulas requires some clarification of how an inference rule is understood. The difference becomes apparent if we compare Modus Ponens in form $A, A \to B/B$ and Modus Tollens $\dashv B, \vdash A \to B/\vdash A$: we cannot write out Modus Tollens without using turnstile and reverse turnstile: if we include the refuted propositions (formulas) into rules, for every premise or conclusion we must indicate whether this proposition is assumed to be asserted or rejected. Let us recall that in \cite{Scott_1974} D.~Scott considers the following possible forms of Modus Ponens\\

\begin{center}
\begin{tabular}{l c l c l c l}
\hline
\multicolumn{7}{c}{Four forms of Modus Ponens}\\ \hline \hline
$A, A \to B \vdash B$ & & $\vdash A \to B$ & & $\vdash A$ & & $\vdash A$\\ \cline{3 - 3}
                      & & $A \vdash B$      & & $\vdash A \to B$ & &$A \vdash B$\\ \cline{5 - 5} \cline{7 - 7}
                      & &                  & & $\vdash B$       & & $\vdash B$ \\ 
(i) & & (ii) & & (iii) & & (iv)\\	\hline		
\end{tabular}    
\end{center}
\begin{center}
Table 1 \label{tbl1}
\end{center}
And he argues that the rule (iii) deserves the name 'Modus Ponens', for "this is methatheoretic statement that the validities of the system are closed under the rule allowing for the detachment of the conclusion of the implication (provided it and its antecedent are valid)." And Scott continues: (i) is a conditional tautology that suggests the rule (iii). Also, he is pointing out that in the {\L}ukasiewicz many-valued logic (i) fails, while (ii),(iii) and (iv) hold. Thus, the rules (i) and (iii) are different. 

It is important to keep in mind that throughout the paper the rules are understood in form (iii). 

The following peculiarity of the derivation systems that admit derivations of asserted propositions from rejected ones is worth mentioning. In the regular deductive systems we can eliminate the rule of Substitution by using axiom schemata instead of axioms. This elimination is possible, because, for example, a derivation like 
\[
\vdash A, \vdash (A \to B), \vdash B, \vdash\sigma(B),
\] 
where $A,B$ are formulas and $\sigma$ is a substitution, can be reduced to 
\[
\vdash \sigma(A), \vdash (\sigma(A) \to \sigma(B)), \vdash \sigma(B)
\] 
and $\sigma(A)$ and $\sigma(A) \to \sigma(B)$ are derivable. The situation changes if we admit the rules containing the rejected propositions as premises and the asserted propositions as conclusions. For example, if we admit the rule 
\begin{equation}
\dashv A, \vdash A \lor B/\vdash B, \label{dpr}
\end{equation} 
the following derivation cannot be reduced like a previous one, even though the rule \eqref{dpr} is structural:
\[
\dashv A, \vdash A \lor B, \vdash B, \vdash \sigma(B),
\]     
because $\dashv \sigma(A)$ may be not derivable. So, we can use the schemata, for instance, in order to define the set of axioms, but we cannot eliminated the rule of Substitution from the deductive system (likewise, we cannot eliminate the rule of Reverse Substitution).

\subsection{What This Paper is Not About} 

There are several very important topics of logic and philosophy of logic that are closely related to the proposed approach. Yet, we will not discuss them, because each of these topics deserves a separate consideration. We will be focusing on studying the meta-logic for rules, and we will not be touching the following topics (that we mention here only in order to underscore their relations to the introduced meta-logic). So, we will not discuss in this paper:

\begin{itemize}

\item The relations between rejection and negation (see, e.g. \cite{Smiley_Rejection_1996,Rumfitt_Yes_2000,Rumfitt_Unilateralism_2002,Gibbard_Price_2002, Restall_Multiple_2005,Rumfitt_Knowledge_2008}). In particular, between $\neg A$ and cases when $\dashv \sigma(A)$ is valid for every substitution $\sigma$ (the reader can consider the formula $\Diamond p \land \neg \Diamond p$, each substitution instance of which is rejected in $\SF$.)  

\item The role and meaning of logical constants (see, e.g. \cite{Prawitz_Natural_1965,Rumfitt_Unilateralism_2002,Dummett_Yes_2002,Dummett_Logical_1991}).

\item The applications to paraconsistent logics (see, e.g. \cite{Paraconsistency_2013}). We just note that there is a big difference between situation when $\vdash A$ and $\vdash \neg A$ are permissible (paraconsistency) and the situation when $\vdash A$ and $\dashv A$ are permissible (we call this \textit{incoherency} or \textit{ambivalence}).

\item The philosophical aspects of multiple-conclusion inference (see, e.g. \cite{Restall_Multiple_2005,Rumfitt_Knowledge_2008}).
 
\item The hypersequent proof systems for rules ( see, e.g. \cite{Iemhoff_Metcalfe_Proof_2009,Iemhoff_Metcalfe_Hypersequent_2009}) .

\item The tableaux methods (see, e.g. \cite{Ishimoto_Axiomatic_1996,Haehnle_Tableaux_2001}).

\item The refutation systems for particular logics (see e.g. \cite{Goranko_Refutation_1994,Skura1998,Sochacki_Book,Skura_Refutation_2011_H, Skura_Book_2013}).
\end{itemize}

\section{Consequence Relations} \label{cons}

In this section we introduce the consequence relations that include simultaneously rejected and asserted propositions. But first and foremost we need to introduce some notions and notation.

\subsection{Language}

Let $\Frm$ be a set of propositional formulas built in a regular way from a countable set of (propositional) variables $\Vars$ and the finite set of (finitely-ary) connectives $f_1,\dots,f_n$ (not containing signs $\mand,\mor,\mimpl,\mneg,\mtop,\mbot,\mpos,\mrej $ that we reserve for use in the meta-language). For propositional variables we will use letters $p,q,r,s$ may be with indexes. To denote the propositional formulas we will use the capital Roman letters (sometimes with indexes) from the beginning of alphabet, while the capital Roman letters (sometimes with indexes) from the end of the alphabet will be used as meta-variables that can be substituted with propositional formulas. For instance, $X \impl Y \land Z$ is a schemata, while $A \impl B \land C$, where $A \bydef p$, $B \bydef q_1 \lor q_2$, $C \bydef \neg r$, is a formula. 

With each propositional variable $p \in \Vars$, we associate a formula variable $X_p$ ranging over $\Frm$. For instance,  $A(p_1,\dots,p_n)$ is a formula on $n$ variables, while $A(X_{p_1},\dots,X_{p_n})$ is a schemata. We will use tilde for denoting schemata, i.e. $\tilde{A}$ is a schemata obtained from a formula $A$ by replacing the propositional variables with corresponding formula variables. Roughly speaking, a schemata is a formula to which we allow to apply substitution. 

A formulas $A(B_1,\dots,B_n)$, obtained from a schemata $\tilde{A} \bydef A(X_{p_1},\dots, X_{p_n})$ (or from a formula $A$ for this matter), by substituting formulas for formula variables, is a \textit{substitution instance} of $\tilde{A}$ (instance, for short).

As usual, a substitution is a mapping $\sigma: \Vars \to \Frm$ and by $\sigma(A)$ we denote a result of simultaneous substitution in $A$ of $\sigma(p)$ for $p$ for every variable occurring in $A$. If $\Gamma$ is a set of formulas and $\sigma$ is a substitution by $\sigma(\Gamma)$ we denote $\set{\sigma(A)}{A \in \Gamma}$. The set of all substitutions will be denoted by $\Sigma$. 

Given a set of formulas $\Gamma$ and a substitution $\sigma$, we say that the set $\Gamma$ is \textit{closed under substitutions}, if $\sigma(\Gamma) \subseteq \Gamma$, and we say that $\Gamma$ is \textit{closed under reverse substitutions}, if $\Gamma \subseteq \sigma(\Gamma)$. It is not hard to see that $\Gamma$ is closed under substitutions if and only if its complement $\Frm \setminus \Gamma$ is closed under reverse substitutions. 

\subsection{Meta-Language: Atomic Statements}

In order to include the rejected propositions into logic, we need to be able, given a formula $A$, to distinguish whether $A$ is asserted or rejected. To achieve this, we use meta-language in which we we can express assertion and rejection. In this section we introduce the metalanguage. 

We start with a notions of statement and schema-statement.

\begin{definition} If $A \in \Frm$ is a formula then the expressions $\mpos A$ and $\mrej A$ we call \textit{atomic statements}. $\mpos A$ is a \textit{positive atomic statement} and $\mpos A$ is a \textit{negative atomic statement}. $A$ is called a \textit{propositional part} of $\mpos A$ or $\mrej A$. By $\mnu A$ we denote an atomic statement with propositional part $A$, that is $\mnu A$ can be $\mpos A$ or $\mrej A$. And by $\notmnu A$ we denote an atomic statement of the "opposite sign", that is $\notmnu A$ is $\mpos A$, if $\mnu A$ is $\mrej A$, and  $\notmnu A$ is $\mrej A$, if $\mnu A$ is $\mpos A$.
\end{definition}

By $\mAStms$ w denote the set of all atomic statements. $\mAStms^+$ and $\mAStms^-$ are respectively the sets of all positive and negative atomic statements. Thus, $\mAStms = \mAStms^+ \cup \mAStms^-$ and $\mAStms^+ \cap \mAStms^- = \emptyset$. 

\begin{remark}[about denotation] J.~{\L}ukasiewicz is using for this purpose $\vdash A$ and $\dashv A$, but we reserve the sign $\vdash$ for consequence relations. In \cite{Smiley_Rejection_1996} T.~Smiley  is using the sign $*$ to denote that a proposition is rejected, that is, $*A$ denotes 'A is rejected'. In \cite{Humberstone_Revival_2000} L.~Humberstone is using $[+]A$ and $[-]A$ and calls 'a signed formula' what we call 'an atomic statement'.
\end{remark}

\subsection{Consequence Relation}

In this section, we extend the notion of consequence relation (e.g. \cite{Scott_completeness_1974}) from sets of formulas to sets of atomic statements. We will use the customary conventions: if $\Gamma, \Delta \subseteq \mAStms$, then $\Gamma,\Delta$ denotes $\Gamma \cup \Delta$ and, if $\alpha \in \mAStms$, then $\Gamma,\alpha$ denotes $\Gamma \cup \{\alpha\}$. 

\begin{definition} A \textit{consequence relation} is a binary relation on the class $\CFin(\mAStms)$ of all finite sets of statements that satisfies the following conditions: 
\begin{itemize}
\item[(R)] if $\Gamma \cap \Delta \neq \emptyset$, then $\Gamma \vdash \Delta$
\item[(M)] if $\Gamma \vdash \Delta$, then $\Gamma, \Gamma_1 \vdash \Delta,\Delta_1 $
\item[(T)] if $\Gamma, \mnu A \vdash \Delta$ and $\Gamma \vdash \mnu A,\Delta$, then $\Gamma \vdash \Delta$,
\end{itemize}
where $\Gamma,\Gamma_1,\Delta,\Delta_1$ are finite sets of atomic statements, $\mnu A$ is an atomic statement.  
\end{definition}

Next, we want to define a notion of structural consequence relation, but the regular definition cannot be used for the following reason: the substitutions into negative statements can lead to unwanted results: if $\vdash \mrej p$ and we allow to substitute any formula for $p$, we will reject every formula. 

\begin{definition} A consequence relation $\vdash$ is called \textit{structural} if for every $\sigma \in \Sigma$
\[
\begin{split}
& \Gamma \vdash \Delta \text{ entails } \sigma(\Gamma) \vdash \sigma(\Delta) \text{ for every } \Gamma,\Delta \subseteq \mAStms^+ \text{ and}\\
& \sigma(\Gamma) \vdash \sigma(\Delta) \text{ entails } \Gamma \vdash \Delta \text{ for every } \Gamma,\Delta \subseteq \mAStms^-.
\end{split}
\]
\end{definition}

It is not hard to see that a meet of structural consequence relations is a structural consequence relation. Hence, for any given set of pairs $\DS \bydef \{\Gamma_i/\Delta_i, i\in I\}$, where $\Gamma_i,\Delta_i \subseteq \mAStms$, there is a smallest consequence relation $\vdash_\DS$ such that $\Gamma_i \vdash_\DS \Delta_i$ for all $i \in I$. We will say that the relation $\vdash_\DS$ is \textit{defined}
by $\DS$.

From this point forward we consider only structural consequence relations. The class of all structural consequence relations is denoted by $\Frm^\vdash$. It is easy to see that an arbitrary meet of consequence relations is a consequence relation. Hence, $\Frm^\vdash$ forms a complete lattice relative to set meet and closed joins, and $\vdash$ defines on $\Frm$ a closure operator (cf. \cite[Proposition 1.1.]{Scott_completeness_1974}). Thus, every set $\class{K} \bydef \{\Gamma_i/\Delta_i\}$, where $\Gamma_i,\Delta_i \subseteq \CFin\mAStms, i\in I$, defines a consequence relation $\vdash_\class{K}$ - the smallest consequence relation such that $\Gamma_i \vdash \Delta_i$, for all $i \in I$. 

\section{Logics} \label{logics}

In this section we introduce the notion of logic that admits the rejected formulas.

\begin{definition} \textit{Logic} is an ordered pair of sets of formulas $\Log{L} = \lbr \LogL^+, \LogL^- \rbr$, where $\LogL^+$ is closed under substitutions and $\LogL^-$ is closed under reverse substitutions. $\LogL^+$ is a \textit{positive} or an \textit{asserted} part of $\Log{L}$, or a set of the \textit{theorems} of $\Log{L}$. $\LogL^-$ is a \textit{negative} or a \textit{rejected} part of $\Log{L}$, or a set of \textit{anti-theorem} of $\Log{L}$. 
\end{definition}

Given a logic $\Log{L}$, by $\Log{L}^+$ and $\Log{L}^-$ we respectively denote a positive and a negative parts of $\Log{L}$.

Every consequence relation $\vdash$ defines a logic $\Log{L}_\vdash \bydef \lbr \LogL^+,\LogL^- \rbr$, where
\begin{equation}
\LogL^+ = \set{A}{A \in \Frm, \vdash \mpos A} \text{ and } \LogL^- = \set{A}{A \in \Frm, \vdash \mrej A}.  \label{logcr}
\end{equation}

\begin{definition}
We use the following terminology: a logic $\Log{L}$ is
\[
\begin{array}{lll}
\text{\textit{coherent}} & \text{:} &\Log{L}^+ \cap \Log{L}^- = \emptyset \\
\text{\textit{full}} & \text{:} &\Log{L}^+ \cup \Log{L}^- = \Frm \\
\text{\textit{standard}} & \text{:} &\text{full and coherent}\\
\text{\textit{trivial}} & \text{:} &\Log{L}^+ = \Log{L}^- = \Frm \\
\text{\textit{degenerate}} & \text{:} &\Log{L}^+ = \Log{L}^- = \emptyset \\
\end{array}
\]
\end{definition}

\begin{prop} Every logic $\Logl$ can be defined by a consequence relation.
\end{prop}
\begin{proof} Indeed, given a logic $\Logl = \lbr \LogL^+, \LogL^-\rbr$, one can consider the following consequence relation: for every finite $\Gamma,\Delta \subseteq \Frm$
\begin{equation}
\Gamma \vdash \Delta \text{ if and only if } \Delta \cap (\Gamma \cup \set{\mpos A}{A \in \LogL^+} \cup \set{\mrej A}{A \in \LogL^-}) \neq \emptyset, \label{mincr}
\end{equation}
that is, we take a trivial consequence relation $\Gamma \vdash \Delta \means \Gamma \cap \Delta \neq \emptyset$ and add as axioms atomic statements obtained form the formulas of $\Log{L}$. It is not hard to see that $A \in \LogL^+$ (or $A \in \LogL^-$) if and only if $\vdash \mpos A$ (or, respectively, $\vdash \mrej A$).
\end{proof}

Any given consequence relation uniquely defines a logic, while the converse is not necessarily true: a given logic $\Log{L}$ may be defined by distinct consequence relations (see Section \ref{remarks} for examples). So, we can consider a class
\[
\vdash_\Log{L} \bydef \set{\vdash \in \Frm^\vdash}{\Log{L} = \Log{L}_\vdash}.
\]  
It is clear that $\vdash_\Log{L}$ is closed under set meet. Hence, $\vdash_\Log{L}$ has a smallest (relative to set inclusion) element, namely the relation defined by \eqref{mincr}. An example of a logic $\Log{L}$ such that $\vdash_\Log{L}$ is not closed under closed joins\footnote{As an example one can consider two calculi defining a logic of 7-element single-generated Heyting algebra: one containing the rule $(\neg \neg p \to p) \to (p \lor \neg p)/(\neg\neg p \lor \neg p)$, and another -- containing the rule $p \lor q/p,q$.}  

In presence of multiple-conclusion rules to define how a formula can be derived from a set of formulas may be somewhat complex (see e.g. \cite{Shoesmith_Smiley_Book}). In order to simplify it, in the following Section we will extend the meta-language endowing it with meta-connectives.  

\subsection{Meta-Language: Statements}

In this section, we enrich the meta-language by introducing statements that play a central role in our research. 

First, we introduce the meta-connectives $\mand, \mor, \mimpl, \mneg, \mtop \mbot$ and we define a notion of statement.

\begin{definition} Atomic statement is \textit{statement}. $\mtop,\mbot$ are statements. If $\alpha,\beta$ are statements then $\alpha \mand \beta, \alpha \mor \beta, \alpha \mimpl \beta, \mneg \alpha$ are statements. The set of all statements we denote by $\Stms$. 
\end{definition}

We will also use tilde to denote \textit{schemata statements}: $\tilde{\alpha}$ denotes the schemata statement obtained from a given statement $\alpha$ by replacing all formulas with schemata. A \textit{substitution instance} (instance, for short) of a given schemata statement $\tilde{\alpha}$ is a statement obtained from $\tilde{\alpha}$ by simultaneous substitution of formula variables with formulas.

Let us observe that statements of type $\mpos A_1 \mand \dots \mand \mpos A_n$ and $\mpos A_1 \mor \dots \mor \mpos A_n$ are Carnap's junctives (cf. \cite[Section D, {\$}21]{Carnap_Formalization_1943}), the former being a conjunctive and the latter being a disjunctive. Also, the schemata statements of form $\mpos \tilde{A_1} \mand \dots \mand \mpos \tilde{A_n} \mimpl \mpos \tilde{B_1} \mor \dots \mor \mpos \tilde{B_m}$ represent the multiple-conclusion rules (comp. \cite[Section 2.3]{Shoesmith_Smiley_Book}).  

Next, we introduce the notions of positive and negative statements.

\begin{definition} A statement that has no occurrences of $\bot$ and negative atomic statements is \textit{positive}. A statement that has no occurrences of $\top$ and positive atomic statements is \textit{negative}. If $\Gamma$ is a set of statements, by $\Gamma^+$ and $\Gamma^-$ we denote the subsets (may be empty) of all positive or, respectively, negative members of $\Gamma$.
\end{definition}
For instance, $\mpos A \mimpl \mpos B$ is a positive statement; $\mrej B$ is a negative statement; a statement $(\mpos A \mimpl \mpos B) \land \mrej B$ is neither positive, nor negative. 

If $\sigma \in \Sigma$ is a substitution, we extend the scope of $\sigma$ from propositional formulas to statements by induction in the following way: 
\[
\begin{array}{ll}
\sigma(\top) \bydef \top , \quad \sigma(\bot) \bydef \bot&  \\
\sigma(\mnu A) \bydef \mnu \sigma(A) & \text{for every } A \in \Frm\\
\sigma(\alpha \circ \beta) \bydef \sigma(\alpha) \circ \sigma(\beta) & \circ \in \{\mand,\mor,\mimpl\} \text{ and for all } \alpha,\beta \in \Stms\\
\sigma(\mneg \alpha) \bydef \mneg \sigma(\alpha) & \text{for all } \alpha \in \Stms\\
\end{array}
\]
In other words, given a statement $\alpha$, $\sigma(\alpha)$ is obtained from $\alpha$ by replacing every occurrence of statement $\mnu A$ with $\mnu \sigma(A)$.   

\section{Deductive System and Inference} \label{inference}

In this section, we extend the notions of deductive system and inference that accommodate the rejected propositions. 

\begin{definition} A \textit{deductive system} is an ordered pair $\lbr \Ax; \Rules \rbr$, where $\Ax \subseteq \Stms$ is a set of \textit{axiom statements} (axioms for short) and $\Rules$ is a set of statement schemata that we call \textit{rules}. The deductive system $\DS_0 \bydef \lbr \emptyset; \emptyset \rbr$ we call a \textit{zero-system}.
\end{definition}

We consider the following three meta-inference rules
\[
\begin{array}{ll}
\text{for every } \alpha,\beta \in \Stms, \text{ from } \alpha \mimpl \beta \text{ and } \alpha \text{ infer } \beta & (MMP)\\
\text{for every } \alpha \in \Stms^+ \text{ and } \sigma \in \Sigma, \text{ from } \alpha \text{ infer } \sigma(\alpha) & (Sb)\\
\text{for every } \alpha \in \Stms^- \text{ and } \sigma \in \Sigma, \text{ from } \sigma(\alpha) \text{ infer } \alpha & (RS)\\
\end{array}
\]

If $\Gamma$ is a set of statements and $\sigma$ is a substitution, we let $\sigma(\Gamma) \bydef \set{\sigma(\alpha)}{\alpha \in \Gamma}$. That is, $\sigma(\Gamma)$ is a set of $\sigma$-substitutions in each member of $\Gamma$. Given a set of statements $\Gamma$, we say that $\Gamma$ is \textit{closed under substitutions} (under Sb, for short), if $\sigma(\Gamma) \subseteq \Gamma$ for every $\sigma \in \Sigma$. And we say that $\Gamma$ is \textit{closed under reverse substitutions}, if $\Gamma \subset \sigma(\Gamma)$ for every $\sigma \in \Sigma$.

The statements, obtained from the axioms of the classical propositional calculus (CPC) (e.g. the axioms of \cite[Group A1]{Kleene_Intro}) by substituting the propositional variables with statements, are called \textit{meta-axioms}.  If $\alpha$ is a statement, then the statements 
\begin{equation*}
\begin{array}{lll}
(\top \mimpl \alpha) \mimpl \alpha & \quad &(\text{Ax}\top) \\
\bot \mimpl \alpha & \quad  & (\text{Ax}\bot)
\end{array} \label{addaxioms}
\end{equation*}
are meta-axioms. 

Now, we can use a regular definition of a (Hilbert style) inference.

\begin{definition} Let $\DS \bydef \lbr \Ax, \Rules \rbr$ be a deductive system, $\Gamma$ be a set of statements and $\alpha$ be a statement. A sequence of statements $\alpha_1,\dots, \alpha_n$ is an \textit{inference} (a \textit{derivation}) of $\alpha$ from $\Gamma$ over $\DS$, if $\alpha_n$ is $\alpha$ and for every $1 \leq i \leq n$one of the following hold
\begin{itemize}
\item[(a)] $\alpha_i$ is a meta-axiom
\item[(b)] $\alpha_i$ is an axiom of $\DS$
\item[(c)] $\alpha_i$ is an instance a rule of $\DS$
\item[(d)] $\alpha_i$ obtained by (Sb) or (RS) from some $\alpha_j$, where $j < i$ 
\item[(e)] $\alpha_i$ obtained by (MMP) from some $\alpha_j,\alpha_k$, where $j,k < i$. 
\end{itemize}
\end{definition}

If there exists an inference of $\alpha$ from $\Gamma$ over $\DS$, we write $\Gamma \vdash_\DS \alpha$ and we say that $\alpha$ \textit{is derivable from $\Gamma$ over} $\DS$. If $\DS$ is a zero-system, we will omit the reference to $\DS$  and we say that $\alpha$ \textit{is derivable from} $\Gamma$. It is easy to see that $\Gamma \vdash \alpha$ if there is a sequence of statements $\alpha_1,\dots,\alpha_n$ such that $\alpha = \alpha_n$ and each $\alpha_i$ is either meta-axiom, or obtained from the preceding statements by (Sb),(Rs) or (MMP).

Obviously, every deductive system $\DS$ defines a consequence relation $\Gamma \vdash_\DS \alpha$ on the finite sets of statements and statements. In its own turn, the consequence relation $\vdash_\DS$ induces a consequence relation on finite sets of atomic statements:
for all $A_1,\dots,A_n,B_1,\dots, B_n \in \Frm$
\begin{equation}
A_1,\dots,A_n \vdash_\DS B_1,\dots, B_n \text{ if and only if } \vdash_\DS A_1 \mand \dots \mand A_n \mimpl  B_1 \mor \dots \mor B_n
\end{equation}
and defines a logic $\Log{L}_\DS$:
\begin{equation}
\Log{L}_\DS^+ \bydef \set{\mpos A}{\vdash_\DS \mpos A, A \in \Frm} \text{ and } \Log{L}_\DS^- \bydef \set{\mrej A}{\vdash_\DS \mrej A, A \in \Frm}. 
\end{equation}
It is not hard to see that the converse is also true: any logic can be defined by some deductive system (which is not necessarily unique). 

\begin{definition}
We say that a deductive system is \textit{full}, \textit{coherent}, \textit{standard} or \textit{trivial} if the logic, defined by this system, is respectively full, coherent, standard or trivial.
\end{definition}

Let us consider an example.

\begin{example} The {\L}ukasiewicz's refutation system for the classical logic (in the signature $\{\to,\neg\}$) can be defined by the deductive system $\lbr \Ax, \Rules \rbr$, consisting of the following four axioms and two rules: 
\[
\begin{array}{ll}
 \mpos ((p \to q) \to ((q \to r)  \to (p \to r))) & \text{axiom (axiom of CPC)}\\
 \mpos((\neg p \to p) \to p) & \text{axiom (axiom of CPC)}\\
 \mpos (p \to (\neg p \to q)) & \text{axiom (axiom of CPC)}\\
 \mrej p & \text{{axiom (\L}ukasiwicz's anti-axiom)}\\
(\mpos X \mand \mpos (X \to Y)) \mimpl  \mpos Y & \text{rule (Modus Ponens)} \\
(\mrej Y \mand \mpos (X \to Y)) \mimpl  \mrej X & \text{rule (Modus Tollens)}
\end{array}
\]
\end{example}
Let us note, that we did not include into our deductive system the rules of substitution and reverse substitution used by {\L}ukasiewicz, because they are already included into the definition of inference.

Let us observe, that (MMP) with the meta-axiom $x \mimpl (y \mimpl (x \mand y))$ allows to apply the inner rules, that is, the inference rules specific for a given logics, like Modus Ponens, for instance. Indeed, suppose we derived the statements $\mpos A$ and $\mpos(A \to B)$. Then, 
\[
\begin{array}{lll}
1. & \mpos A & \text{was derived} \\
2. & \mpos (A \to B) & \text{was derived} \\
3. & \mpos A \mimpl (\mpos(A \to B) \mimpl (\mpos A \mand \mpos(A \to B))) & \text{a substitution instance of } \\
&& x \mimpl (y \mimpl (x \mand y)) \\
4. & \mpos A \mand \mpos(A \to B) & \text{from } 1,2 \text{ and } 3 \text{ by (MMP)}\\
5. & (\mpos A \mand \mpos (A \to B)) \mimpl  \mpos B & \text{an instance of Modus Ponens}\\
6. & \mpos B & \text{from 4  and 5 by (MMP)}
\end{array}
\]

As usual, two statements $\alpha$ and $\beta$ are equivalent if $\vdash \alpha \mimpl \beta$ and $\vdash \beta \mimpl \alpha$. 

\begin{prop} \label{conjrules} Every statement $\alpha$ is equivalent to a meta-conjuction of the statements of form $\mnu A_1 \mand \dots \mand \mnu A_n \mimpl \mnu B_1 \mor \dots \mor \mnu B_m$. 
\end{prop}
\begin{proof} In the classical logic, every formula is equivalent to a formula in conjunctive normal form. Hence, every statement $\alpha$ is equivalent to a meta-conjunction of statements, each of which has the following form
\begin{equation}
\mneg \mnu A_1 \mor \dots \mor \mneg \mnu A_n \mor \mnu B_1 \mor \dots \mor \mnu B_m. \label{prop2_1}  
\end{equation}
It is clear that\eqref{prop2_1} is equivalent to
\begin{equation}
\mneg (\mnu A_1 \mand \dots \mand \mnu A_n) \mor \mnu B_1 \mor \dots \mor \mnu B_m \label{prop2_2}  
\end{equation}
and \eqref{prop2_2} is equivalent to
\begin{equation*}
(\mnu A_1 \mand \dots \mand \mnu A_n) \mimpl \mnu B_1 \mor \dots \mor \mnu B_m.   
\end{equation*}
\end{proof}

From the Proposition \ref{conjrules}, it follows that $\alpha$ is equivalent to a meta-conjunction of rules for any given deductive system $\DS$ as long as for every $\sigma \in \Sigma, \vdash_\DS \sigma(\alpha)$.  

\section{Concluding Remarks} \label{remarks}

As mentioned at the beginning of the paper, in \cite{Smiley_Rejection_1996} T.~Smiley has introduced a calculus for CPC contains rules for rejection proposition. Among these rules he uses the following\footnote{We are using the notation from the present paper.}
\begin{equation}
\begin{array}{ll}
\mrej A \vdash \mpos \neg A & \quad \quad  (r1)\\
\mpos \neg A \vdash \mrej A & \quad \quad   (r2)
\end{array}
\end{equation} Let us observe that the above rules yield that $\vdash \mrej A$ entails that for every substitution $\sigma \in \Sigma, \vdash\mrej \sigma(A)$, i.e. substitutions preserve validity of rejection. Indeed, from $\vdash \mrej A$ and the rule (r1) we infer $\vdash \mpos \neg A$. Since $\mpos \neg A$ is an assertion, for every substitution  $\sigma \in \Sigma$ we have $\vdash \neg \sigma(A)$ and, using rule (r2), we get $\vdash \mrej \sigma(A)$. Thus, the {\L}ukasiewicz's axiom for rejection $\mrej p$, where $p$ is a propositional variable, is not valid in the Smiley's logic: $\vdash \mrej p$ yields $\mrej A$ for every formula $A$, while Smiley's logic is not trivial. In other words, the Smiley's logic does not admit the rule of reverse substitution. So, in the Smiley's logic the formula $p$ is neither asserted, nor rejected.

Let us remark that the above considerations can be applied to every logic in which rejection of a formula is equal to assertion of negation of this formula. 

Let us also note that the use of multiple-conclusion rules gives us an ability to construct distinct consequence relations defining standard classical propositional logic. Indeed, one can add the rule $p \lor q/p,q$ to the {\L}ukasiewicz's calculus and the obtained calculus will still define the classical logic. In fact, there is infinite set of distinct consequence relations defining the classical logic: for every $k > 1$ we can add to the {\L}ukasiewicz's calculus the rule 
\[
R_k \bydef \top/\{(p_i \eqv p_j), i\neq j, 1 \leq i,j \leq 2^k \}. 
\]
It is not hard to see that the rule $R_k$ is valid in a Boolean algebra with $2^k$ elements, but is not valid in any Boolean algebra with more than $2^k$ elements. Thus, the calculi obtained from the {\L}ukasiewicz's calculus by adding the rules $R_k$ are distinct and non-trivial.   

In conclusion, we note that the presence of the rule of reverse substitution makes semantic of such logics much more complex and calls for the use of matrices similar to Q-matrices introduced by G.~Malinowski (see, e.g. \cite{Malinowski_Q_Consequence_1991}). We will discuss the semantic for the logic with rejection in a separate paper.  

\section*{Acknowledgments} The author wishes to express his gratitude to A.~Muravitsky, the discussions with whom helped during the work on this paper.


\bibliographystyle{acm}

\def\cprime{$'$}

\end{document}